\documentclass[a4j,12pt]{amsart}

\usepackage{amsmath}
\usepackage{amsbsy}
\usepackage{amssymb}
\usepackage{amscd}
\usepackage{euscript}

\makeatletter
\renewcommand{\c@equation}{\c@thm}
\makeatother

\usepackage[dvips]{graphicx, color}
\usepackage{slashbox}

\newcommand{\bbq}{{\mathbb Q}}
\newcommand{\bbr}{{\mathbb R}}

\newcommand{\bbz}{{\mathbb Z}}

\newcommand{\Del}{{\Delta}}

\newcommand{\im}{{\operatorname {Im}}}

\newcommand{\Z}{\bbz}
\newcommand{\Q}{\bbq}
\newcommand{\R}{\bbr}

\newcommand{\beq}{\begin{equation}\begin{aligned}}
\newcommand{\eeq}{\end{aligned}\end{equation}}
\newcommand{\beginp}{\begin{proof}[証明]}

\newcommand{\beqn}{\begin{eqnarray}}
\newcommand{\eeqn}{\end{eqnarray}}

\newcommand{\op}{\operatorname}

\newcommand{\beqnn}{\begin{eqnarray*}}
\newcommand{\eeqnn}{\end{eqnarray*}}
\newcommand{\beqnna}{\begin{eqnarray}\begin{array}}
\newcommand{\eeqnna}{\end{array}\end{eqnarray}}
\newcommand{\bsub}{\begin{subarray}}
\newcommand{\esub}{\end{subarray}}

\newtheorem{thm}{Theorem}[section]
\newtheorem{lem}[thm]{Lemma}
\newtheorem{cor}[thm]{Corollary}
\newtheorem{prop}[thm]{Proposition}

\theoremstyle{definition}

\newtheorem{defn}[thm]{Definition}

\newtheorem{rem}[thm]{Remark}        

\title{Lower bounds of the canonical height on quadratic twists of 
elliptic curves}
\author{Tadahisa Nara}

\date{}


\begin{document}

\begin{abstract}
%
%
We compute a lower bound of the canonical height 
on quadratic twists of certain elliptic curves. 
Also we show a simple method for 
constructing families of quadratic twists 
with an explicit rational point. 
Using the above lower bound, 
we show that the explicit rational point is primitive 
as an element of the Mordell--Weil group. 
\end{abstract}

\maketitle

\section{Introduction}
It is known that 
for every elliptic curve, 
there exists a positive lower bound 
of the canonical heights of non-torsion rational points (\cite{sil2}). 
There is also an algorithm which computes a lower bound 
for a given elliptic curve (\cite{cs}). 
 
In the paper \cite[Proposition 8.3]{duquesne1}, 
Duquesne gave an explicit lower bound of 
the canonical heights of raional points 
on a certain family of elliptic curves. 
The family consists of quartic twists of the elliptic curve 
$y^2=x^3-x$. 
Similarly Fujita and the auther gave an explicit lower bound 
on a family consisting of 
sextic twists of the elliptic curve $y^2=x^3+1$ (\cite{fn1}). 
Both results are used to show 
that some explicit points can always be in a system of generators of 
the Mordell--Weil groups. 

In this paper we give 
an explicit lower bound 
for a family consisting of 
quadratic twists of an elliptic curve. 
There is already a non-explicit bound 
(\cite[Exercise 8.16]{aec}) 
given by a different method from ours (see Remark \ref{rem:O(1)}). 
Making the bound explicit 
enables us to study explicitly the behavior of 
a certain family of the quadratic twists of a given elliptic curve. 
For example, we can proove Theorem \ref{thm:D013} below.  

Our lower bound is computed by 
using 
the decomposition of the canonical height into the local heights 
and they are computed by 
the combination of Cohen's algorithm (\cite[Algorithm 7.5.7]{cohen0}) 
and Silverman's algorithm (\cite[Theorem 5.2]{sil1}). 
In \cite{duquesne1} and \cite{fn1}, by the simplicity of the forms of 
the Weierstrass equations, the estimates of the non-archimedean part 
of the local height were given by ad hoc arguments. 
However, in our case 
more systematic argument is required. 
The key is an identity between division polynomials 
of elliptic curves (Lemma \ref{eq:id}).

Our main results are as follows. 

\begin{thm}
\label{main}
Let $E/\Q$ be an elliptic curve defined by $y^2=x^3+a_2x^2+a_4x+a_6$ 
$(a_2, a_4, a_6 \in \Z)$ with the discriminant $\Delta$. 
Let $D$ be a square-free integer 
and $E_D/\Q$ the elliptic curve 
$y^2=x^3+a_2Dx^2+a_4D^2x+a_6D^3$. 
Assume that $\Delta$ is 6th-power-free. 
Then for $P \in E_D(\Q)\setminus E_D(\Q)[2]$, 
\begin{align*}
\hat{h}(P) > \frac{1}{4}\log |D| 
+\frac{1}{16} \log \frac{(1-|q|)^8}{|q|}
+\frac1{4}\log \left|\frac{\omega}{2\pi}\right|
-\frac7{16}\sum_{p|\Delta,p\neq2}\log p-\frac5{12}\log 2, 
\end{align*}
where 
$\omega_1$ and $\omega_2$ are periods of $E$ such that $\omega_1>0,\ 
\op{Im}(\omega_2)>0$ and $\op{Re}(\omega_2/\omega_1)=0$ or $-1/2$, 
$q=\exp(2\pi i{\omega_2}/{\omega_1})$ and  
\begin{eqnarray*}
\omega=
\left\{ \begin{array}{ll}
\omega_1  & (D>0) \\
\im(\omega_2) & (D<0, \Delta>0) \\
2\im(\omega_2) & (D<0, \Delta<0) 
 \end{array} \right..
\end{eqnarray*}

%
%
\end{thm}

\begin{rem}
\label{rem:O(1)}
We have 
$\hat{h}(P) > \frac{1}{4}\log |D| +O(1)$ 
by \cite[Exercise 8.16 (c)]{aec}. 
The proof does not use the local height functions. 
Note that our $\hat{h}$ is twice of 
that in \cite{aec}, \cite{cohen0} and \cite{sil1}.
 
\end{rem}
$E_D$ in  Theorem \ref{main}
is called the quadratic twist of $E$ by $D$, 
which is isomorphic  over $\Q$ to the curve defined by 
$Dy^2=x^3+a_2x^2+a_4x+a_6$. 

Using Theorem \ref{main}, 
we can also show the following theorem. 

\begin{thm}
\label{thm:D013}
Let $t \in \Z$, $D(t)=t^6+4t^4+30t^3+5t^2+54t+245$, 
$E_D$ the elliptic curve $y^2=x^3+2D(t)x^2+163D(t)^2x+2205D(t)^3$ and  
$P$ the point $( D(t)(t^4+2t^2+12t)$, $D(t)^2(t^3+t+3) )$ on $E_D$. 
We assume that $D(t)$ is square-free. 
Then $P$ is a primitive point if $|t|\geq 2216$. 
In particular $E_D(\Q) \simeq \langle P \rangle $ if 
$\op{rank}E_D(\Q)=1$. 
\end{thm}

\begin{rem}
This family of quadratic twists 
is an example given by the method described in Section \ref{sec:fam_quad}. 
For many other families given by the method, 
we can show similar results. 
\end{rem}

The organization of this paper is as follows. 
In Section \ref{sec:preliminaries} 
we review the notions of 
the canonical height and the local height function. 
In Section \ref{sec:fam_quad} 
we introduce a method of constructing 
families of quadratic twists. 
In Section \ref{sec:unif lower} 
we compute the local height functions 
by using Cohen's algorithm and 
Silverman's algorithm 
to prove Theorem \ref{main}. 
In Section \ref{sec:comp-canon-height} 
we prove Theorem \ref{thm:D013}, which 
is a consequence for an example given by 
the method in Section \ref{sec:fam_quad}. 
%
%
%

\section{Preliminaries}
\label{sec:preliminaries}
Let $E$ be an elliptic curve 
$y^2+a_1xy+a_3y=x^3+a_2x^2+a_4x+a_6$. 
Throughout this paper $b_2,b_4,b_6$ and $c_4,c_6$ 
are the usual quantities defined in \cite[III.1]{aec}. 
Further by $\Delta$, 
we denote the discriminants of $E$. 
If we have to specify the elliptic curve, 
we may use the notation such as $\Delta_E$. 

First, we define the notion of 
the canonical height of elliptic curves. 
Let $E/\Q$ be an elliptic curve and $P=(x,y) \in E(\Q)$ 
with $x=n/d$ and $\gcd(n,d)=1$. 
Then the na\"{i}ve height $h(P)$ is defined by 
$\max \{\log |n|, \log |d|\}$ 
and the canonical height $\hat{h}(P)$ is defined by 
\begin{equation*}
\hat{h}(P)=\lim_{n \rightarrow \infty} \frac{h(2^nP)}{4^n}. 
\end{equation*}

%
%

It is known that the canonical height 
is decomposed to the sum of functions, 
called the local height functions. 
We use the decomposition for computations of the canonical heights. 
The local height function 
${\lambda}_v$ is defined by the following theorem. 
\begin{thm}
$($N$\acute{e}$ron, Tate, \cite[p. 341]{sil1}$)$
\label{thm:Neron}
Let $K$ be a number field, $v$ a place and 
$K_v$ its completion with respect to 
the absolute value $|\cdot|_v$.
Let $E/K$ be the elliptic curve 
$y^2+a_1xy+a_3y=x^3+a_2x^2+a_4x+a_6$. 
Then there exists a unique function 
${\lambda}_v : E(K_v)\setminus{O} \rightarrow \R$
which has the following three properties. 
\begin{itemize}
\item[(1)]
For all P $\in E(K_v)$ with $2P \neq O$, 
\begin{equation*} 
{\lambda}_{v}(2P)=4{\lambda}_{v}(P)-
2\log |2y(P)+a_1x(P)+a_3|_v. 
\end{equation*}
\item[(2)]
The limit 
\begin{math} 
\lim_{\substack{P\rightarrow O \\ \text{$v$-adic}}}
({\lambda}_v(P)-\log |x(P)|_v) 
\end{math}
exists.  
\item[(3)]
${\lambda}_v$ is 
bounded on any $v$-adic open subset of $E(K_v)$ disjoint from $O$.
\end{itemize}
\end{thm}

\begin{rem}
\label{rem:shift}
There is an alternative definition of the local height function, 
which is given by adding $\frac12 \log|\Delta|$ 
on the right hand side of (1) 
(\cite[Chapter VI, Theorem 1.1]{aecadv}). 
This alternative local height function is independent 
of the Weierstrass equation. 
With our definition, 
the local height function 
depends on the Weierstrass equation, 
but the function does not change 
by the substitution $x \mapsto x+r$ (see \cite[Lemma 2.11]{fn1}), 
which corresponds to the shift of the Weierstrass model 
in the direction of $x$-axis. 

The definition of the local height function in 
Cohen's algorithm (\cite[Algorithm 7.5.7]{cohen0}), 
which we shall use later in this paper, 
agrees with ours except for the multiplication by 1/2. 
\end{rem}
%
%
%
%
%
%
Now if $K=\Q$ we have the decomposition 
\begin{align}
\label{decomp}
\hat{h}(P)=\sum_{p:\op{prime},\infty}{\lambda}_p(P). 
\end{align}

%
%

\section{Families of quadratic twists}
\label{sec:fam_quad}
In this section we describe a method to construct 
families of quadratic twists of elliptic curves 
with an explicit point. 

Let $f \in \Z[t]$ be a monic irreducible cubic polynomial 
(therefore with no multiple roots), 
$F \in \Z[t]$ a polynomial such that $F'=mf$ for some $m\in \Z$ 
and 
$\alpha$ a root of $f$. 
The minimal polynomial of $F(\alpha)$ over $\Q$ 
is a cubic polynomial, which is denoted by $f_1$. 
Then $f_1 \circ F(t)$ has the factor $f(t)^2$, since 
$f_1 \circ F(\alpha)=0$ and $\cfrac{d (f_1\circ F)}{dt}(\alpha)$ 
$=f_1'(F(\alpha))F'(\alpha)$$=0$. 
Therefore, there exists a polynomial $D(t)$ 
such that  $D(t)f(t)^2=f_1(F(t))$. 

For example, if 
\begin{equation*}
f=t^3+t+3, \quad F=t^4+2t^2+12t,
\end{equation*}
we have 
\begin{equation*}
f_1(x)=x^3+2x^2+163x+2205, \quad
D(t)={t}^{6}+4\,{t}^{4}+30\,{t}^{3}+5\,{t}^{2}+54\,t+245
.\end{equation*}
So we have the quadratic twist $D(t)y^2=f_1(x)$ 
of the elliptic curve $y^2=f_1(x)$, 
and it has the obvious rational point $(F(t),f(t))$. 

If $h$ is a polynomial, we denote its
discriminant by $\op{disc}(h)$.

\begin{lem}
Let $A,B\in\Z$, $f=t^3+At+B$ and $F=t^4+2At^2+4Bt$.
Then the polynomials 
$f_1$ and $D$ as above are as follows. 
\begin{align*}
f_1=t^3+2A^2t^2+A(A^3+18B^2)t+B^2(2A^3+27B^2), \\
D=t^6+4At^4+10Bt^3+5A^2t^2+18ABt+2A^3+27B^2. 
\end{align*}
In particular, $\op{disc}(f_1)=B^2 \op{disc}(f)^3$. 
\end{lem}

\begin{proof}
If we write $f(t)=(t-\alpha_1)(t-\alpha_2)(t-\alpha_3)$, 
then 
\begin{equation*}
f_1(t)=(t-F(\alpha_1))(t-F(\alpha_2))(t-F(\alpha_3)). 
\end{equation*}
Since 
\begin{align*}
& F(\alpha_1)+F(\alpha_2)+F(\alpha_3), \\
& F(\alpha_1)F(\alpha_2)+F(\alpha_2)F(\alpha_3)+F(\alpha_3)F(\alpha_1), \\
& F(\alpha_1)F(\alpha_2)F(\alpha_3) 
\end{align*}
are all symmetric polynomials of 
$\alpha_1, \alpha_2, \alpha_3$, they are polynomials of 
$\alpha_1+ \alpha_2+ \alpha_3 \, (=0)$, 
$\alpha_1\alpha_2+\alpha_2\alpha_3+\alpha_3\alpha_1 \, (=A)$ and 
$\alpha_1\alpha_2\alpha_3 \, (=-B)$. 
It is easy to verify that 
\begin{align*}
& F(\alpha_1)+F(\alpha_2)+F(\alpha_3)=-2A^2, \\
& F(\alpha_1)F(\alpha_2)+F(\alpha_2)F(\alpha_3)+F(\alpha_3)F(\alpha_1)
=A(A^3+18B^2), \\
& F(\alpha_1)F(\alpha_2)F(\alpha_3)=-B^2(2A^3+27B^2). 
\end{align*}

\end{proof}

\begin{rem}
Let $E'$ be the elliptic curve  $y^2=f_1(x)$. 
Then 
\begin{equation*}
\Delta_{E'}=16\op{disc}(f_1)=16B^2\op{disc}(f)^3
=16B^2(-4A^3-27B^2)^3
.\end{equation*}
So for example, if $B$ is odd, 
$\gcd(A,B)=1$ and $\op{disc}(f)$ is square-free, 
then Theorem \ref{main} is applicable to 
$Dy^2=f_1(x)$. 
\end{rem}

\section{Uniform lower bound on quadratic twists}
\label{sec:unif lower}

In this section we compute a lower bound of the canonical height 
on quadratic twists of elliptic curves. 
%
%
%
%
%
%
%
%
%
%
%
We use the decomposition (\ref{decomp}). 

Consider an elliptic curve of the form
\begin{equation}
\label{eq:elliptic-curve-main}
E:y^2=x^3+a_2x^2+a_4x+a_6
\end{equation}
where $a_2,a_4,a_6 \in\Z$ (the point is that $a_1=a_3=0$). 
For a square-free integer $D$ 
we put 
\begin{equation}
\label{eq:E_D}
E_D : y^2=x^3+a_2Dx^2+a_4D^2x+a_6D^3. 
\end{equation}
Throughout this section, 
by $\omega_1$ and $\omega_2$ 
we denote 
the periods of $E$ 
such that $\omega_1>0,\ 
\op{Im}(\omega_2)>0$ and $\op{Re}(\omega_2/\omega_1)=0$ or $-1/2$ 
and put $q=\exp (2\pi i \omega_2/\omega_1)$. 
The periods, discriminant and the usual quantities of $E_D$ are 
denoted by 
$\omega_{1,D}$, $\Delta_D$, $a_{i,D}$, $b_{i,D}$ and $c_{i,D}$. 

Straightforward computations using 
\cite[Algorithm 7.4.7]{cohen0} show
the following lemma. 
\begin{lem}
We have 
$\omega_{1,D}=\omega |D|^{-1/2}$, 
where 
\begin{eqnarray*}
\omega=
\left\{ \begin{array}{ll}
\omega_1  & (D>0) \\
\im(\omega_2) & (D<0, \Delta>0) \\
2\im(\omega_2) & (D<0, \Delta<0) 
\end{array} \right..
\end{eqnarray*}
\end{lem}

%
%
%

We first consider the archimedean part, 
that is ${\lambda}_{\infty}$. 
Points in $E_D(\Q)$ can always be made into the form 
$(\alpha/\delta^2,\beta/\delta^3)$, where 
$\alpha,\beta,\delta \in \Z$, 
$\delta>0$ and 
$\gcd(\alpha,\delta)=\gcd(\beta,\delta)=1$. 
So in this sectition we always assume the above condition on 
$\alpha, \beta$ and $\delta$. 
\begin{lem}
\label{lem:arc}
Let $Q=(\alpha/\delta^2, \beta/\delta^3) 
\in E_D(\Q)\setminus E_D(\Q)[2]$. 
Then 
\begin{equation}
\begin{aligned}
\lambda_{\infty}(Q) & \geq \frac{1}{4}\log |D| 
+\frac{1}{16} \log \frac{(1-|q|)^8}{|q|}
+ \frac1{4}\log \left|\frac{\omega}{2\pi}\right| \\
& \quad -\frac32\log \delta 
+\frac1{2}\log |{\beta}|
+\frac{1}{16} \log |\Delta|. 
\end{aligned}
\end{equation}
\end{lem}
\begin{proof}
By \cite[Algorithm 7.5.7]{cohen0} 
and the trivial bound $|\theta|\leq 1/(1-|q|)$, 
\begin{align*}
\lambda_{\infty}(Q)&=\frac{1}{16} \log \left|\frac{\Delta_D}{q}\right|
+\frac1{4}\log \left|\left(\frac{\beta}{\delta^3}\right)^2 
\frac{\omega_{1,D}}{2\pi}\right|
-\frac1{2}\log|\theta| \\
&=\frac{1}{16} \log \left|\frac{\Delta}{q}\right|
+\frac6{16}\log |D|
+\frac1{4}\log \left|\frac{\omega}{2\pi}\right|
+\frac1{4}\log \left|\frac{\beta^2}{\delta^6}\right|
-\frac14 \log |D|^{\frac12}
-\frac1{2}\log|\theta| \notag\\
&\geq \frac1{4}\log |D| 
+\frac{1}{16} \log \left|\frac{\Delta}{q}\right|
+\frac1{4}\log \left|\frac{\omega}{2\pi}\right|
+\frac1{2}\log \left|\frac{\beta}{\delta^3}\right|
-\frac1{2}\log \frac1{1-|q|} 
\notag\\
&=\frac1{4}\log |D| 
+\frac{1}{16} \log \frac{(1-|q|)^8}{|q|}
+\frac1{4}\log \left|\frac{\omega}{2\pi}\right|
-\frac32\log \delta 
+\frac1{2}\log |\beta|
+\frac{1}{16} \log |\Delta|. 
\end{align*}
\end{proof}

\begin{rem}
Note that we can not use \cite[Algorithm 7.5.7]{cohen0} 
for 2-torsion points. 
\end{rem}

To prove Theorem \ref{main}, 
we shall consider a lower  bound of 
the sum of the last two terms 
(i.e. $\frac1{2}\log |\beta|
+\frac{1}{16} \log |\Delta|$) 
and the 
non-archimedean part. 
To compute the non-archimedean part of the canonical height, 
we use Silverman's algorithm (\cite[Theorem 5.2]{sil1}). 
\begin{defn}
For an elliptic curve defined by 
$E : y^2+a_1xy+a_3y=x^3+a_2x^2+a_4x+a_6$ 
we define polynomials of $x,y$ as follows.
\begin{align*}
\psi_{0}(x,y)&=3x^2+2a_2x+a_4-a_1y, \\
\psi_{2}(x,y)&=2y+a_1x+a_3, \\
\psi_{2a}(x,y)&=4x^3+ 2b_2x^2+b_4x+b_6, \\
\psi_{3}(x,y)&=3x^4+b_2x^3+3b_4x^2+3b_6x+b_8. 
\end{align*}
\end{defn}

Apart from $\psi_0$, 
they are known as the division polynomials of elliptic curves. 
For a point $Q=(x_0,y_0)$, we put $\psi_i(Q)=\psi(x_0,y_0)$. 
Note that if $Q \in E$, $\psi_2(Q)^2=\psi_{2a}(Q)$.  

Before computing the division polynomials of $E_D$, 
we compute the usual quantities 
of the Weierstrass equation of $E$ 
in (\ref{eq:elliptic-curve-main}) 
as follows.  Since $a_1=a_3=0$ in our case, 
we have 
\begin{align}
\label{Delta}
&\Delta=-16\,\left( 27\,{a}_{6}^{2}-18\,{a}_{2}\,{a}_{4}\,{a}_{6}+4\,{a}_{2}^{3}\,{a}_{6}+4\,{a}_{4}^{3}-{a}_{2}^{2}\,{a}_{4}^{2}\right), \\
&\label{c4}
c_4=-16\,\left( 3\,{a}_{4}-{a}_{2}^{2}\right), \\
&\label{c6}
c_6=-32\,\left( 27\,{a}_{6}-9\,{a}_{2}\,{a}_{4}+2\,{a}_{2}^{3}\right).
\end{align}
Note that 
\begin{align}
\label{id:Delta_c4_c6}
\Delta=1728^{-1}(c_4^3-c_6^2)=2^{-6} 3^{-3}(c_4^3-c_6^2). 
\end{align}

%
%
The usual quantities of the Weierstrass equation of $E_D$ 
are as follows. 
\begin{align*}
& a_{1,D}=a_{3,D}=0, 
a_{2,D}=a_2D,\ a_{4,D}=a_4D^2,\ a_{6,D}=a_6D^3, \\
& b_{2,D}=4a_2D,\ b_{4,D}=2a_4D^2,\ b_{6,D}=4a_6D^3,\ 
b_{8,D}=(4a_2a_6-a_4^2)D^4,\\
& c_{4,D}=16(a_2^2-3a_4)D^2,\ 
c_{6,D}=-32\,( 27\,{a}_{6}-9\,{a}_{2}\,{a}_{4}+2\,{a}_{2}^{3})D^3,\ 
\Delta_D=\Delta D^6. 
\end{align*}

Using this, we have the division polynomials 
of $E_D$ as follows. 
\begin{align*}
\psi_{0,D}(x,y)&=3x^2+2a_2Dx+a_4D^2, \\
\psi_{2,D}(x,y)&=2y, \\
\psi_{2a,D}(x,y)&=4x^3+ 2b_{2,D}x^2+b_{4,D}x+b_{6,D} \\
&=4(x^3+a_2Dx^2+a_4D^2x+a_6D^3), \\
\psi_{3,D}(x,y)&=3x^4+b_{2,D}x^3+3b_{4,D}x^2+3b_{6,D}x+b_{8,D}\\ 
&=3x^4+4a_{2}Dx^3+6a_{4}D^2x^2+12a_{6}D^3x+(4a_2a_6-a_4^2)D^4. 
\end{align*}


For $E_D$, $Q (\in E_D)$ and $p$ we put 
\begin{align*}
A&=v_p(\psi_{0,D}(Q)), \\
B&=v_p(\psi_{2,D}(Q))=v_p(\psi_{2a,D}(Q))/2, \\
C&=v_p(\psi_{3,D}(Q)). 
\end{align*}
Roughly speaking, 
by comparing these values for $Q$, 
the output of Silverman's algorithm 
becomes the value of $\lambda_p(Q)$. 
%
%
%
%


Even though the following lemma follows from 
direct computations, 
it plays a key role in the subsequent computations. 

%
%

\begin{lem}
\label{eq:id}
Let 
\begin{align*}
k_{D}(x,y)&:=3{x}^{2}+2{a}_{2}Dx
+(4{a}_{4}-{a}_{2}^{2}){D}^{2},\notag \\
l_{D}(x,y)&:=9{x}^{3}+9{a}_{2}D{x}^{2}
+(21{a}_{4}-4{a}_{2}^{2}){D}^{2}x
+(27{a}_{6}-2{a}_{2}{a}_{4}){D}^{3}. \notag
\end{align*} 
Then 
\begin{align}
\label{id3,2a}
-16 k_{D}\cdot \psi_{3,D}+4 l_{D}\cdot\psi_{2a,D}=-\Delta D^6. 
\end{align}
\end{lem}

In the following consideration, we fix a 
square-free integer $D$ and a rational point
$Q=(x_0,y_0)=(\alpha/\delta^2, \beta/\delta^3)\in E_D(\Q)$.

\begin{defn}
\label{omega}
Let $\Omega$ be the set of all the rational primes. 
We put 
\begin{align*}
&\Omega^+=\{p\in \Omega ; p|\delta\}, \ 
\Omega^-=\{p\in \Omega ; p{\not|} \delta\}, \\
&\Omega_1=\{p\in \Omega^-\setminus\{2,3\} ; p{\not|} \Delta, p{\not|} D \}, \ 
\Omega_2=\{p\in \Omega^-\setminus\{2,3\} ; p{\not|} \Delta, p{|} D \}, \\
&\Omega_3=\{p\in \Omega^-\setminus\{2,3\} ; p{|} \Delta, p{\not|} D \}, \ 
\Omega_4=\{p\in \Omega^-\setminus\{2,3\} ; p{|} \Delta, p{|} D \}. 
%
\end{align*}

\end{defn}
%

%
If $p \not\in \Omega^+$, 
then $v_p(x_0) \geq 0$ and so 
\begin{equation*}
v_p(k_D(x_0,y_0)) \geq 0, \quad
v_p(l_D(x_0,y_0))\geq 0, \quad
v_p(\psi_{i,D}(x_0,y_0)) \geq 0.
\end{equation*}
%
%
%
To ease the notations, 
we put 
\begin{align*}
&k_D=k_D(Q)=k_D(\alpha/\delta^2,\beta/\delta^3), \\
&l_D=l_D(Q)=l_D(\alpha/\delta^2,\beta/\delta^3), \\
&\psi_{i,D}=\psi_{i,D}(Q)=\psi_{i,D}(\alpha/\delta^2,\beta/\delta^3). 
\end{align*}

\begin{defn}
\begin{itemize}
\item[(1)] 
For a set of primes ${S}$ and an integer $m$, we define 
$m_{{S}} = \prod_{p \in {S}} p^{v_p(m)}$. 
\item[(2)] 
We put $\Lambda=\lambda_p(Q)/\log p$ and $N=v_p(\Delta)$
(here we are considering $\Del=\Del_E$ and not $\Del_D$). 
\end{itemize}
\end{defn}

Now we compute $\lambda_p(Q)$ 
using \cite[Theorem 5.2]{sil1} 
in Lemmas \ref{lem:o+}--\ref{lem:p=3}. 
Recall in our definition the value of $\hat{h}$ 
is twice of that in \cite{sil1}, and so is $\lambda_p$. 
We assume that $\Delta$ is 6th-power-free. 
So (\ref{eq:E_D}) is a minimal Weierstrass equation 
at every prime $p$ by \cite[VII, Remark 1.1]{aec}. 
\begin{lem}
\label{lem:o+}
If $p \in \Omega^+$ and $Q=(\alpha/\delta^2,\beta/\delta^3)$
$\in E_D(\Q)$, 
then $\sum_{p\in \Omega^+}\lambda_p (Q) 
= 2 \log \delta$. 
\end{lem}

\begin{proof}
Since $p{|}\delta$, $p{\not|}\alpha$ and $p{\not|}\beta$. 
So the reduction of $Q$ modulo $p$ is 
nonsingular. 
Therefore 
\begin{equation*}
\sum_{p\in \Omega^+}\lambda_p (Q) 
= \sum_{p\in \Omega^+}\max \{0, -v_p(\alpha/\delta^2)\}\log p 
=2 \log \delta.
\end{equation*}
\end{proof}

\begin{lem}
\label{lem:P1}
If $p \in \Omega_1$ and $Q=(\alpha/\delta^2,\beta/\delta^3)$
$\in E_D(\Q)$, 
then $\lambda_p (Q) = 0$. 
\end{lem}

\begin{proof}
Since, $p {\not|}\Delta_D$, the image of $Q$ under the reduction modulo $p$ is 
nonsingular. 
Therefore 
$\lambda_p (Q)= \max \{0, -v_p(\alpha/\delta^2)\}\log p=0$. 
\end{proof}

\begin{lem}
\label{lem:P2}
If $p \in \Omega_2$ and $Q=(\alpha/\delta^2,\beta/\delta^3)$ 
$\in E_D(\Q)$, 
then $\lambda_p (Q) \geq -\log p$ and $v_p(\beta)\geq2$. 
In particular $\sum_{p \in \Omega_2}\lambda_p (Q) 
+\frac1{2}\log |{\beta_{\Omega_2}}| 
\geq 0$. 
\end{lem}

%
%

\begin{proof}
To consider a lower bound of $\lambda_p$, we may assume 
$p|\beta$ (so $p{\not|}\delta$), since otherwise $\lambda_p(Q)=0$. 
Recall $\psi_{2,D}(Q)^2=\psi_{2a,D}(Q)$. 
Since $p|\psi_{2,D}(Q)$, $p|\psi_{2a,D}(Q)$. 
So $p$ has to divide $\alpha$. 
Then $v_p(\psi_{2a,D})\geq 3$. 
On the other hand $v_p(\psi_{2a,D})$ is even, 
and so 
$v_p(\psi_{2a,D})\geq 4$ and $B=v_p(\beta)\geq 2$. 
Clearly $v_p(k_{D})\geq2$ and $v_p(l_{D})\geq3$. 
So we have $v_p(\psi_{3,D})\leq 4$ by (\ref{id3,2a}). 
Then since $3B>C$, $\Lambda = -C/4\geq -1$. 
(Note that $p|c_{4,D}$ and so the additive reduction occurs ). 

\end{proof}

\begin{lem}
\label{lem:P3}
If $p \in \Omega_3$ and $Q=(\alpha/\delta^2,\beta/\delta^3)$ 
$\in E_D(\Q)$, then 
\begin{equation*}
\lambda_p (Q) 
+\frac1{2}\log |{\beta_{\{p\}}}| 
+\frac1{16}\log|\Delta_{{\{p\}}}|
\geq -\frac1{12}\log |\Delta_{{\{p\}}}|.
\end{equation*}
\end{lem}

\begin{proof}
Note that 
\begin{align*}
\lambda_p (Q) 
+\frac1{2}\log |{\beta_{\{p\}}}| 
+\frac1{16}\log|\Delta_{{\{p\}}}|
=\left(\Lambda+\frac{B}{2}+\frac{N}{16}\right) \log p. 
\end{align*}
At first, we assume that $p|c_4$. 
Then $E_D$ has the additive reduction at $p$. 
By (\ref{id:Delta_c4_c6}), 
$N=v_p(\Delta)=2,3$ or $4$ since $\Delta$ is 6th-power-free.
By (\ref{id3,2a}) 
$\min \{v_p(\psi_{2a,D}), v_p(\psi_{3,D}) \}$ $\leq$ $v_p(\Delta)$. 
So we rewrite this inequality as $\min\{2B,C\}\leq N$. 
If $(3B >) 2B>C$, 
then 
\begin{align*}
\Lambda=-\frac{C}{4},\ \frac{B}{2}>\frac C{4},\ \frac N{16} \geq \frac C{16}. 
\end{align*} 
So 
\begin{align*}
\Lambda+\frac B{2}+\frac N{16} \geq \frac C{16} \geq 0. 
\end{align*}
If $3B> C \geq 2B$ (therefore $3B\geq C+1$), 
then 
\begin{align*}
\Lambda=-\frac{C}{4} \geq -\frac{3B-1}{4},\ \frac N{16} \geq \frac{2B}{16}. 
\end{align*}
So 
\begin{align*}
\Lambda+\frac B{2}+\frac N{16} \geq -\frac B{8}+\frac 1 4 
\geq -\frac 2 8+ \frac 1 4=0. 
\end{align*}
If $C \geq 3B (>2B)$, 
then 
\begin{align*}
\Lambda=-\frac{2B}{3},\ \frac N{16} \geq \frac{2B}{16}.
\end{align*} 
So 
\begin{align*}
\Lambda+\frac B{2}+\frac N{16} \geq -\frac B{24} 
\geq -\frac 2{24}=-\frac 1{12}. 
\end{align*}

Next we assume that $p{\not|}c_4$. 
Then the multiplicative reduction occurs at $p$.
Then $\Lambda=-\frac{n(N-n)}{N}$, 
where $n=\min\{B,N/2\}$. Since $N \leq 5$, $-\frac{n(N-n)}{N}\geq -6/5$. 
So if $B\geq3$, clearly $\Lambda+B/2 \geq 0$. 
For the cases $B=1,2$, 
by case-by-case argument, we can verify that 
$\Lambda+B/2+N/16 
\geq 0$ for all the case of $N=1,2,3,4,5$. 
\end{proof}

\begin{lem}
\label{lem:P4}
If $p \in \Omega_4$ and $Q=(\alpha/\delta^2,\beta/\delta^3)$ 
$\in E_D(\Q)$, then 
\begin{equation*}
\lambda_p (Q)
+\frac1{2}\log |{\beta_{\{p\}}}| + \frac1{16}\log|\Delta_{\{p\}}|
\geq -\frac7{16} \log p.
\end{equation*}
\end{lem}

\begin{proof}
Since $p|c_{4,D}$, the additive reduction occurs at $p$. 
We may assume $p|\beta$ as in Lemma \ref{lem:P2}, 
and so $p|\alpha$. 
By (\ref{id3,2a}), we have $\min \{ C+2,2B+3\}\leq N+6$. 

If $3B>C$ (therefore $3B\geq C+1$) 
and $2B+3>C+2$ (therefore $2B\geq C$), 
then 
\begin{align*}
\Lambda =-\frac{C}{4},\ 
\frac B{2} \geq \frac{C}{4},\ 
\frac N{16}>\frac{C-4}{16}. 
\end{align*}
So 
\begin{align*}
\Lambda+\frac B{2}+\frac N{16} \geq \frac{C-4}{16} \geq -\frac 1 4. 
\end{align*}
If $3B>C$ (therefore $3B\geq C+1$) 
and $C+2 \geq 2B+3$, 
then 
\begin{align*}
\Lambda =-\frac{C}{4}\geq -\frac {3B-1}{4},\ 
\frac N{16} \geq \frac{2B-3}{16}. 
\end{align*}
So 
\begin{align*}
\Lambda+\frac B{2}+\frac N{16} 
\geq -\frac B{8}+\frac 1 {16} \geq -\frac 4 8+\frac 1{16}=-\frac 7{16}. 
\end{align*}
If $C \geq 3B$, then $C+2 \geq 2B+3$. 
So 
\begin{align*}
\Lambda =-\frac{2B}{3},\ 
\frac N{16} \geq \frac {2B-3}{16}.
\end{align*} 
Therefore 
\begin{align*}
\Lambda+\frac B{2}+\frac N{16} 
\geq -\frac B{24}-\frac 3{16} \geq -\frac 4{24}-\frac 3{16}=-\frac{17}{48}. 
\end{align*}
\end{proof}

\begin{lem}
\label{lem:p=2}
If $p =2$ and $Q=(\alpha/\delta^2,\beta/\delta^3)$ 
$\in E_D(\Q)$, then
\begin{equation*}
\lambda_2 (Q) 
+\frac1{2}\log |\beta_{ \{ 2 \} }| 
\geq -\frac 2{3} \log 2.
\end{equation*}
In particular, 
\begin{equation*}
\lambda_2 (Q) 
+\frac1{2}\log |\beta_{ \{ 2 \} }|
+ \frac1{16}\log|\Delta_{\{2\}}|
\geq -\frac 5{12} \log 2.
\end{equation*}
\end{lem}

\begin{proof}
By (\ref{c4}) and (\ref{c6}), we can write 
$v_2(c_4^3)=3k\ (k\geq4)$
and $v_2(c_6^2)=2l\ (l\geq5)$. 
So by (\ref{id:Delta_c4_c6}), $v_2(\Delta)=4$ since $\Delta$ is 
6th-power-free, since $\Delta \in \Z$. 
Note that $2|c_{4,D}$ and so the additive reduction occurs. 

If $2{\not|}\beta$, $B=1$. Then $\Lambda=-2/3$ or $-1/2$. 
Therefore, $\Lambda \geq -2/3$. 

If $2|\beta$ and $2|D$, then we may assume $2|\alpha$ 
to consider a lower bound of $\lambda_2$, 
for otherwise $A=0$. 
By the same argument as that in Lemma \ref{lem:P2}, 
we have $v_2(\psi_{2a,D})\geq 6$ and $v_2(\beta)\geq2$. 
Similarly by the identity (\ref{id3,2a}), we have 
$v_2(\psi_{3,D})\leq 4$ and so $\Lambda\geq-1$. 

If $2|\beta$ and $2{\not|}D$, then we may assume $v_2(\psi_{3,D})\geq 1$, 
for otherwise $C=0$ and $\Lambda=0$. 
By the identity (\ref{id3,2a}), 
we have $v_2(\psi_{2a,D})\leq 2$ (actually $v_2(\psi_{2a,D})= 2$). 
Then $\Lambda=-2/3$ or $-1/2$.  Therefore, $\Lambda \geq -2/3$. 
\end{proof}

\begin{lem}
\label{lem:p=3}
If $p =3$ and $Q=(\alpha/\delta^2,\beta/\delta^3)$ 
$\in E_D(\Q)$, then
\begin{equation*}
\lambda_3 (Q)
+\frac1{2}\log |{\beta_{\{3\}}}| + \frac1{16}\log|\Delta_{\{3\}}|
\geq -\frac7{16} \log 3.
\end{equation*}
\end{lem}

\begin{proof}
We may assume $3|\Delta_D$. 
If $3{\not|}\Delta$ and $3{|}D$, 
then by the same argument as that in Lemma \ref{lem:P2}, 
$\lambda_3 (Q)
+\frac1{2}\log |{\beta_{\{3\}}}|\geq 0$. 
If $3{|}\Delta$ and $3{\not|}D$, 
then in the case of $3|c_4$ we cannot deny the possibility of $N=5$, 
but anyway $B\leq2$. 
So by the same argument as that in Lemma \ref{lem:P3}, 
\begin{equation*}
\lambda_3 (Q)
+\frac1{2}\log |{\beta_{\{3\}}}| + \frac1{16}\log|\Delta_{\{3\}}|
\geq -\frac1{12} \log 3. 
\end{equation*}
If $3{|}\Delta$ and $3{|}D$, 
then by the same argument as that in Lemma \ref{lem:P4}, 
\begin{equation*}
\lambda_3 (Q)
+\frac1{2}\log |{\beta_{\{3\}}}| + \frac1{16}\log|\Delta_{\{3\}}|
\geq -\frac7{16} \log 3.
\end{equation*}
\end{proof}

We now finish the proof of Theorem \ref{main}. 

\begin{proof}[proof of Theorem \ref{main}]
By (\ref{decomp}) and Lemma \ref{lem:arc} 
\begin{align*}
\hat{h}(P) \geq \frac{1}{4}\log |D| 
+&\frac{1}{16} \log \frac{(1-|q|)^8}{|q|}
+\frac1{4}\log |\frac{\omega_{1}}{2\pi}|\\
-&\frac32\log \delta 
+\frac1{2}\log |{\beta}|
+\frac{1}{16} \log |\Delta|
+\sum_{p:\text{prime}}\lambda_p(Q). 
\end{align*}
If $2,3 \not\in \Omega^+$, then 
\begin{align*}
\frac1{2}\log |{\beta}|
+&\frac{1}{16} \log |\Delta|
+\sum_{p:\text{prime}}\lambda_p(Q) \\
&= 
\sum_{S=\Omega^+, \Omega_1, \Omega_2, \Omega_3, \Omega_4, \{2\},\{3\}} 
\left(\frac1{2}\log |{\beta}_S|
+\frac{1}{16} \log |\Delta_S|
+\sum_{p\in S}\lambda_p(Q) \right)
\end{align*} 
and a lower bound of the right hand side is given by Lemmas 
\ref{lem:o+}-\ref{lem:p=3}, 
that is 
$2\log \delta -\frac7{16}\sum_{p|\Delta,p\neq2}\log p-\frac5{12}\log 2$. 
Even if $2$ or $3$ $\in \Omega^+$, the same bound is valid, 
since the lower bounds given in Lemmas \ref{lem:p=2} and \ref{lem:p=3} 
are negative. 
%
\end{proof}


\begin{cor}
\label{cor:low013}
Let 
$E_D$ be the elliptic curve $y^2=x^3+2Dx^2+163D^2x+2205D^3$ $(D>0)$ 
and 
$Q \in E_D(\Q)\setminus E_D(\Q)[2]$. Then 
\begin{align}
\hat{h}(Q)>\frac14\log D -3.5472. 
\end{align}
\end{cor}

\begin{proof}
Let $E$ be the elliptic curve defined by 
$y^2=x^3+2x^2+163x+2205$. Then we have 
$\Delta=-2^4 3^2 13^3 19^3$, 
$\omega_1=1.04995090\cdots$,    
$q=-0.10978666\cdots     
$ 
%
%
by the following commands in PARI/GP v.2.3.4 
and the corollary follows. 
\begin{verbatim}
E=ellinit([0,2,0,163,2205]); 
factor(E.disc)
E.omega[1]
exp(2*Pi*I*E.omega[2]/E.omega[1])
\end{verbatim}
\end{proof}

\section{An example}
\label{sec:comp-canon-height}

In this section we consider a family of 
quadratic twists of an elliptic curve. 
The family in the following lemma is 
constructed by the method described in Section \ref{sec:fam_quad}. 

\begin{lem}
\label{P013-arc}
Let $t \in \Z$, $D(t)=t^6+4t^4+30t^3+5t^2+54t+245$, 
$E_D$ the elliptic curve defined by 
$y^2=x^3+2D(t)x^2+163D(t)^2x+2205D(t)^3$ and  
$P$ the point $( D(t)(t^4+2t^2+12t)$, $D(t)^2(t^3+t+3) )$ on $E_D$. 
Then 
\begin{equation*}
{\lambda}_{\infty}(P)<\frac5{3}\log D(t) + 1.2177. 
\end{equation*}
\end{lem}
\begin{proof}
We fix an integer $t$. 
We transform the Weierstrass equation 
by $x \mapsto x-30D(t)$. This yields the 
elliptic curve defined by 
\begin{equation*}
y^2=x^3-88D(t)x^2+2743D(t)^2x-27885D(t)^3
\end{equation*}
and $P$ corresponds to the point 
$(D(t)(t^4+2t^2+12t+30)$, $D(t)^2(t^3+t+3))$. 
We denote them by $E_D'$ and $P'$ respectively. 
Now ${\lambda}_{\infty}(P)$ on $E_D$ equals 
${\lambda}_{\infty}(P')$ on $E_D'$ (Remark \ref{rem:shift}).

The polynomial 
\begin{math}
x^3-88x^2+2743x-27885
\end{math}
has only one real root, which we denote by $c$, 
and its approximate value is 
$20.55166\cdots$ (this can easily be found by softwares like Maple).  
So the  only real root of $x^3-88D(t)x^2+2743D(t)^2x-27885D(t)^3$
is $c D(t)$, and so we have $x(Q)>20.55166D(t)>0$ for $Q \in E_D'(\R)$ 
(it is easy to see that $D(t)>0$ for $t \in \R$).  
So ${\lambda}_{\infty}(P')$ is computable by using Tate's series 
as follows. 
\begin{align*}
{\lambda}_{\infty}(P')=\log |x(P')|
+ \sum_{i=0}^{\infty}\frac1{4^{i+1}}\log |z(2^iP')|, 
\end{align*}
where 
\begin{align*}
z(P')=1-\frac {b_{4,D}}{x(P')^2}
-\frac {2b_{6,D}}{x(P')^3}-\frac {b_{8,D}}{x(P')^4}.
\end{align*}
Note that for any $Q \in E_D'(\R)$ 
we have $z(Q)>0$, 
since $z(Q)$ satisfies the 
equaliy $z(Q)x(Q)^4=\psi_2(Q)^2x(2Q)$. 

By elementary calculus, 
we can compute the bounds of the series as follows. 
\begin{align*}
0<\frac{x(P')}{D(t)^{5/3}}=
\frac{t^4+2t^2+12t+30}{(t^6+4t^4+30t^3+5t^2+54t+245)^{2/3}}
<3.37933 ,
\end{align*}
So 
\begin{align*}
\log x(P')<\frac5{3}\log D(t)+\log(3.37933)
=\frac5{3}\log D(t)+1.217674.
\end{align*}
For any point $Q \in E_D'(\R)$, there exists 
$u (>20.55166)$ such that $x(Q)=D(t)u$. 
So 
\begin{align*}
z(Q)=1-\frac {b_{4,D}}{x(Q)^2}
-\frac {2b_{6,D}}{x(Q)^3}-\frac {b_{8,D}}{x(Q)^4}
=1-\frac{5486}{{u}^{2}}+\frac{223080}{{u}^{3}}
-\frac{2291471}{{u}^{4}} <1.
\end{align*}
Therefore 
\[\sum_{i=0}^{\infty}\frac1{4^{i+1}}\log z(2^iP')<0.\] 
\end{proof}

\begin{lem}
\label{P013-nonarc}
We consider the situation of Lemma \ref{P013-arc}, 
and assume that $D(t)$ is square-free. 
Then we have 
\begin{align*}
\sum_{p:{\rm prime}}\lambda_p(P) \leq-\log D(t). 
\end{align*}
\end{lem}

\begin{proof}
To ease the notation, we write $D(t)=D$. 
Since the discriminant of $E_D$ is $\Delta_{D}=D^6\Delta
=-D^6\cdot2^4 3^2 13^3 19^3$ 
and $D$ is square-free, 
$E_D$ is a minimal Weierstrass equation.  
Since $P$ is an integral point, 
$\lambda_p(P)$ is non-positive for every $p$ 
and so 
\begin{align*}
\sum_{p:{\rm prime}}\lambda_p(P) 
\leq \sum_{p|D}\lambda_p(P). 
\end{align*}
%
%
%
%

$E_D$ has the additive reduction at $p$ dividing $D$. 
By the definition of $\psi_{2a,D}$ and $\psi_{3,D}$, 
it is clear that
$v_p(\psi_{2a,D}(P))\geq3$ and $v_p(\psi_{3,D}(P))\geq4$. 
Note that $v_p(\psi_{2a,D}(P))$ is even and so 
$v_p(\psi_{2a,D}(P))\geq4$. 
So for $p$ such that $p|D$ 
we have 
$\lambda_p(P)\leq -\frac44\log p$ or 
$\lambda_p(P)\leq -\frac43\log p$ 
by Silverman's algorithm. 
In any cases $\lambda_p(P)\leq -\log p$ and 
so 
\begin{equation*}
\sum_{p|D}\lambda_p(P)\leq \sum_{p|D}(-\log p) =-\log D.
\end{equation*}
\end{proof}

Lemmas \ref{P013-arc}, \ref{P013-nonarc} imply the
following proposition. 

\begin{prop}
In the situation of Lemma \ref{P013-nonarc},  
we have 
\begin{align*}
\hat{h}(P)< \frac23\log D(t)+1.2177.
\end{align*}
\end{prop}

We now finish the proof of Theorem \ref{thm:D013}.

\begin{proof}[proof of Theorem \ref{thm:D013}]
Since $D(t)^2(t^3+t+3)\neq 0$, 
$P$ is not a 2-torsion point, 
and so by Corollary \ref{cor:low013} 
not a torsion point for $|t|\geq 11$. 
By elementary calculus, we have 
\begin{align*}
\frac{\frac23\log D(t) +1.2177}{\frac14\log D(t) -3.5472}<4, 
\end{align*}
for $|t|\geq 2216$. 
Therefore by the property of the canonical height, 
there does not exists a point $R$ such that $P=mR$ 
($|m|\geq 2$). 

\end{proof}


\bibliographystyle{jplain}
\bibliography{RefM}

\end{document}